\documentclass[a4paper,11pt]{article}
\usepackage{amsfonts,amsmath,amsthm,amssymb,pdfsync}
\usepackage[mathcal]{euscript}
\usepackage[latin1]{inputenc}
\usepackage[dvips]{graphicx}
\newtheorem{teo}{Theorem}

\newcommand\R{{\mathbb R}}
\newcommand\F{{\mathcal F}}
\newcommand\LL{{\mathcal L}}
\DeclareMathOperator{\dive}{div}

\author{Giuseppe Buttazzo  \thanks{ Dipartimento di Matematica, Largo B. Pontecorvo 5, I-56127, Pisa, Italy. Email  address:
buttazzo@dm.unipi.it.} \ \ \
Faustino Maestre \thanks{%
Departamento de Ecuaciones Diferenciales y Análisis Numérico,
Facultad de Matemáticas, Universidad de Sevilla, Tarfia s/n,
E-41012 Sevilla, Spain. Email address: fmaestre@us.es.}
}

\title{Optimal Shape for Elliptic Problems with Random Perturbations}

\begin{document}

\maketitle

\begin{abstract}
In this paper we analyze the relaxed form of a shape optimization problem with state equation
$$\left\{\begin{array}{ll}
-\dive\big(a(x)Du\big)=f\qquad\hbox{in }D\\
\hbox{boundary conditions on }\partial D.
\end{array}\right.$$
The new fact is that the term $f$ is only known up to a random
perturbation $\xi(x,\omega)$. The goal is to find an optimal
coefficient $a(x)$, fulfilling the usual constraints $\alpha\le
a\le\beta$ and $\displaystyle\int_D a(x)\,dx\le m$, which
minimizes a cost function of the form
$$\int_\Omega\int_Dj\big(x,\omega,u_a(x,\omega)\big)\,dx\,dP(\omega).$$
Some numerical examples are shown in the last section, to stress the difference with respect to the case with no perturbation.
\end{abstract}

\section{Introduction}\label{intro}

The field of shape optimization problems received in the last
years a particular attention from the mathematical community, also
in view of the many possible applications in high-tech instruments
and structures, where increasing the performances or decreasing
the weight, even by a small percentage, could be crucial. Several
books on the field have been written, exploring the various
aspects (theoretical, numerical, modelling, \dots) that intervene
in this very rich subject; we quote for instance \cite{Allaire}, \cite{sigmund}, \cite{bubu}, \cite{Cherkaev}, \cite{HP}, \cite{milton}, \cite{sozo}, \cite{Tartar_2000}.

The general framework of a shape optimization problem is the following: given a bounded domain $D$ of $\R^d$ and a given right-hand side $f$, for every subdomain $A\subset D$ a PDE
$$E_Au=f$$
is considered, with given boundary data. The PDE above produces a unique solution $u_A$ which, inserted into an integral cost function, provides the final cost
$$F(A)=\int_Dj\big(x,u_A(x),Du_A(x)\big)\,dx.$$
The shape optimization problem, under a volume constraint on the class of admissible choices, is then
$$\min\big\{F(A)\ :\ A\subset D,\ |A|\le m\big\}.$$

Due to a strong instability of the class of domains, very often an optimal shape does not exist, and the optimization problem is usually {\it relaxed} into a more treatable form, where the main unknown is the coefficient of an elliptic PDE on the whole set $D$. In the present paper we consider the simplest case, where the PDE is of a linear elliptic type
\begin{equation}\label{pde}
-\dive\big(a(x)Du\big)=f\qquad\hbox{in }D
\end{equation}
with the boundary conditions on $\partial D$  of Dirichlet type
$$u=u_0\hbox{ on }\partial D.$$

The new fact is that the right-hand side $f$ in \eqref{pde} is only known up to a random perturbation; more precisely, if $(\Omega,\F,P)$ is a probability space, we assume that
$$f(x,\omega)=f(x)+\xi(x,\omega),$$
where the random perturbation $\xi$ is such that
$$\int_\Omega\xi(x,\omega)\,dP(\omega)=0\qquad\hbox{for a.e. }x\in D.$$

There are few references in the literature of this sort of random
or stochastic optimal design problems; a general $\Gamma$-convergence framework was introduced in \cite{dmmo}, while an optimal design problem in a finite dimensional setting was considered in \cite{Alvarez-Carrasco}.

The homogenization method (see \cite{Allaire}, \cite{Tartar_2000}) and
the classical tools of non-convex variational problems (in
particular, Young measures, see \cite{Pgal_first_article}, \cite{Pgal_libro}) are the two mostly used approaches in the mathematical literature to analyze optimal design problems. We will use the Homogenization Theory in order to obtain the existence of a solution and some necessary conditions of optimality.

In the last section we consider some simple cases of loads
$f(x,\omega)$ and we perform a numerical analysis of the optimal
configurations, showing the differences between the deterministic
case $f(x)$ and the perturbed one $f(x)+\xi(x,\omega)$.

\section{The optimization problem}\label{secpb}

We consider a bounded open set $D\subset\R^d$ with a Lipschitz
boundary, two constants $\alpha$ and $\beta$ such that
$0<\alpha\le\beta$, and a given value $m\in(\alpha|D|,\beta|D|)$.

We also consider a probability space $(\Omega,\F,P)$ and a random map $f:\Omega\to L^2(D)$ that we write as
$$f(x,\omega)=f(x)+\xi(x,\omega),$$
where $\xi$ has the property
$$\int_\Omega\xi(x,\omega)\,dP(\omega)=0\qquad\hbox{for a.e. }x\in D.$$
For every coefficient $a(x)$ verifying
$$\alpha\le a(x)\le\beta,\qquad\int_Da(x)\,dx\le m$$
we consider the linear elliptic PDE
\begin{equation}\label{eqstate}
\left\{\begin{array}{ll}
-\dive\big(a(x)Du\big)=f(x,\omega)\\
u=0\hbox{ on }\partial D,%\qquad\frac{\partial u}{\partial
%n}=0\hbox{ on }\Gamma_1
\end{array}\right.
\end{equation}
which provides a unique solution $u_a(x,\omega)$. Finally, we consider a cost functional of the form
\begin{equation}\label{cost}
F(a)=\int_\Omega\Big[\int_Dj\big(x,\omega,u_a(x,\omega)\big)\,dx\Big]\,dP(\omega)
\end{equation}
where $j(x,\omega,u)$ is measurable, l.s.c. in $u$, and such that
for suitable $c>0$ and $\Lambda\in L^1(D\times\Omega)$
$$j(x,\omega,u)\ge\Lambda(x,\omega)-c|u|^2\qquad\forall(x,\omega,u).$$
More general cost functionals, of the form
$$F(a)=\int_\Omega\Big[\int_Dj\big(x,\omega,u_a(x,\omega),Du_a(x,\omega)\big)\,dx\Big]\,dP(\omega)$$
could also be considered, but we limit ourselves to the simpler case \eqref{cost}, having in mind the energy
$$j(x,\omega,u)=-f(x,\omega)u$$
and the compliance
$$j(x,\omega,u)=f(x,\omega)u.$$
The optimization problem we consider is
\begin{equation}\label{optpb}
\min\Big\{F(a)\ :\ \alpha\le a(x)\le\beta,\ \int_Da(x)\,dx\le m\Big\}.
\end{equation}
Note that the optimal coefficient we look for is {\it deterministic}, that is it does not depend on the random variable $\omega$.

Besides to problem \eqref{optpb} we will consider, mainly for the numerical purposes, the penalized one
\begin{equation}\label{penpb}
\min\Big\{F(a)+\lambda\int_Da(x)\,dx\ :\ \alpha\le a(x)\le\beta\Big\}
\end{equation}
where $\lambda$ is the Lagrange multiplier of the mass constraint.

\section{The state equation}

The approach we follow in order to prove the existence of solution of the optimization problem \eqref{optpb} consists in checking that our problem can be seen as a relaxed optimal design problem of another auxiliary optimal design problem and from this relaxed character we deduce the existence of a solution. We focus on the Homogenization Method. Throughout this section, we denote by $\chi_n\in L^\infty(D;\{0,1\})$, $n=1,2,...$, a sequence of characteristic functions and $A_n\in M^{d\times d}$ a sequence of tensors of the form:
$$A_n(x)=\alpha\chi_n(x) I_d+\beta(1-\chi_n(x)) I_d$$
with $0<\alpha\le\beta$.

\subsection{The Homogenization Method}

The homogenization method is based on the concept of $H$-convergence (see \cite{Allaire}, \cite{murat_tartar1}, \cite{murat_tartar2}, \cite{murat_tartar3}). We say that a sequence of tensors $\{A_n\}_{n\in N}$ $H$-converges to the tensor $A_\ast\in L^\infty(D,M^{n\times n})$ if, for any $f$ such that $f(\cdot,\omega)\in H^{-1}(D)$ $P$-a.e. $\omega\in\Omega$, the sequence $\{u_n\}$ of solutions of
\begin{equation}\nonumber
\left\{
\begin{array}{rcll}
-\dive\big(A_n(x)\nabla u_n(x,\omega)\big)&=&f(x,\omega)&{\rm in}\ D\\
u_n&=&0&{\rm on}\ \partial D.
\end{array}
\right.
\end{equation}
satisfies
\begin{equation}\nonumber
\left\{\begin{array}{l}
u_n(\cdot,\omega)\rightharpoonup u(\cdot,\omega)\quad\textrm{in }
H^1_0(D),\qquad\textrm{$P$-a.e. }\omega\in\Omega\\
A_n\nabla u_n(\cdot,\omega)\rightharpoonup A_\ast\nabla
u(\cdot,\omega)\quad\textrm{in }L^2(D)^d,\qquad\textrm{$P$-a.e. }\omega\in\Omega
\end{array}\right.
\end{equation}
where $u(\cdot,\omega)$ is the solution of the homogenized equation $P$-a.e. $\omega\in\Omega$
\begin{equation}\nonumber
\left\{
\begin{array}{rcll}
-\dive\big(A_\ast(x)\nabla u(x,\omega)\big)&=&f(x,\omega)&{\rm in}\ D,\\
u&=&0&{\rm on}\ \partial D.
\end{array}
\right.
\end{equation}
We shall write $A_n\overset{H}{\longrightarrow}A_\ast$ to indicate
this kind of convergence.

We consider $A,B\in M^{d\times d}$, $\{\chi_n\}_{n\in
N}\subset L^\infty(D,\{0,1\})$ a sequence of characteristics
functions, $\{A_n\}_{n\in N}$ the sequence of matrices
\begin{equation}\nonumber\label{def:lam_homo_1}
A_n(x)=\chi_n(x)A+(1-\chi_n(x))B.
\end{equation}
We assume (which always occurs for a subsequence) that there exist
$\theta\in L^\infty(D,[0,1])$ and $A_\ast\in L^\infty(D,M^{d\times d})$
such that
\begin{equation}\nonumber\label{def:lam_homo_2}
\chi_n(x)\overset{\ast}{\rightharpoonup}\theta(x)\quad\textrm{in }L^\infty(D,[0,1])
\end{equation}
and
\begin{equation}\nonumber\label{def:lam_homo_3}
A_n\overset{H}{\longrightarrow}A_\ast.
\end{equation}
In this case $A_\ast$ is called the homogenized tensor obtained by the composition of the two phases $A$ and $B$, in proportions $\theta$ and $1-\theta$ respectively, and with the microstructure defined by the sequence $\{\chi_n\}_{n\in N}$.

In this sense the homogenized tensor $A_\ast$ is characterized by three components, the phases $A$ and $B$ and the proportion $\theta$. Therefore an important issue is to identify all possible homogenized tensors once fixed these three components, this is the so-called $G$-closure problem.

Fortunately, for the case of two isotropic matrices, the $G$-closure in the deterministic case is well known (see \cite{Allaire}, \cite{luce}, \cite{murat_tartar1}, \cite{murat_tartar3}). We will prove that our ``random'' $G$-closure remains equal to the deterministic one. We denote $G_\theta$ and $\tilde{G}_\theta$ the $G$-closure associated with the deterministic and random equations.

\begin{teo}\label{teo:carac_Gteta_autovalores}
Given $\theta\in L^\infty(D;[0,1])$ the $G$-closure of the two
isotropic tensors $\alpha I_d$ and $\beta I_d$ with proportions
$\theta$ and $(1-\theta)$ respectively, is the set of symmetric
matrices with eigenvalues $\lambda_1,\lambda_2,...,\lambda_d$ such
that,
$$\lambda_\theta^-\le\lambda_i\le\lambda_\theta^+\qquad 1\le i\le d$$
$$\sum_{i=1}^d\frac{1}{\lambda_i-\alpha}\le
\frac{1}{\lambda_\theta^--\alpha}+\frac{d-1}{\lambda_\theta^+-\alpha}$$
$$\sum_{i=1}^d\frac{1}{\beta-\lambda_i}\le\frac{1}{\beta-\lambda_\theta^-}+\frac{d-1}{\beta-\lambda_\theta^+}$$
where $\lambda_\theta^+$ and $\lambda_\theta^-$ are the arithmetic
and harmonic means of $\alpha$ and $\beta$ with proportions $\theta$,
$$\lambda_\theta^-=\Big(\frac{\theta}{\alpha}+\frac{1-\theta}{\beta}\Big)^{-1}
\qquad\textrm{and}\qquad\lambda_\theta^+=\theta\alpha+(1-\theta)\beta.$$
\end{teo}

\begin{proof} We will prove that $G_\theta=\tilde{G}_\theta$.

We start proving that $G_\theta\subset\tilde{G}_\theta$ and we
consider $A_\ast\in G_\theta$. Therefore there exists a sequence
of matrices $\{A_n\}_{n\in N}$ of the form
$\alpha\chi_nI_d+\beta(1-\chi_n)I_d$ with
$\chi_n\overset{\ast}{\rightharpoonup}\theta$, such that for every
right-hand side $f(x)$ the solutions $u_n$ of
$$\left\{
\begin{array}{rcll}
-\dive\big(A_n(x)\nabla u_n(x)\big)&=&f(x)&\textrm{in }D\\
u_n&=&0&\textrm{on }\partial D
\end{array}\right.$$
satisfy
$$\left\{
\begin{array}{l}
u_n\rightharpoonup u\quad\textrm{in }H^1_0(D),\\
A_n\nabla u_n\rightharpoonup A_\ast\nabla u\quad\textrm{in }L^2(D)^d,
\end{array}\right.$$
where $u$ is the solution of the homogenized equation
$$\left\{
\begin{array}{rcll}
-\dive\big(A_\ast(x)\nabla u(x)\big)&=&f(x)&\textrm{in }D\\
u&=&0&\textrm{on }\partial D.
\end{array}\right.$$
It is enough to observe that when $f(x)$ is replaced by $f(x,\omega)$ the above convergence holds $P$-a.e. $\omega\in\Omega$, and therefore $A_\ast\in\tilde{G}_\theta$.

We prove now that $\tilde{G}_\theta\subset G_\theta$ and we take $A_\ast\in\tilde{G}_\theta$. Therefore there exists a sequence of matrices $\{A_n\}_{n\in N}$ of the form $\alpha\chi_nI+\beta(1-\chi_n)I$ with $\chi_n\overset{\ast}{\rightharpoonup}\theta$, such that for every right-hand side $\tilde{f}(x,\omega)$ the solutions $\tilde{u}_n$ of
\begin{equation}\label{eq:ran1}
\left\{
\begin{array}{rcll}
-\dive\big(A_n(x)\nabla\tilde{u}_n(x,\omega)\big)&=&\tilde{f}(x,\omega)&\textrm{in }D\\
\tilde{u}_n&=&0&\textrm{on }\partial D
\end{array}
\right.
\end{equation}
satisfy
\begin{equation}\label{eq:ran2}
\left\{
\begin{array}{l}
\tilde{u}_n(\cdot,\omega)\rightharpoonup\tilde{u}(\cdot,\omega)
\quad\textrm{in }H^1_0(D),\qquad\textrm{$P$-a.e. }\omega\in\Omega\\
A_n\nabla\tilde{u}_n(\cdot,\omega)\rightharpoonup
A_\ast\nabla\tilde{u}(\cdot,\omega)\quad\textrm{in }L^2(D)^d,\qquad\textrm{$P$-a.e. }\omega\in\Omega
\end{array}
\right.
\end{equation}
where $\tilde{u}(\cdot,\omega)$ is the solution of the homogenized
equation $P$-a.e. $\omega\in\Omega$
\begin{equation}\label{eq:ran3}
\left\{
\begin{array}{rcll}
-\dive\big(A_\ast(x)\nabla \tilde{u}(x,\omega)\big)&=&\tilde{f}(x,\omega)&\textrm{in }D\\
\tilde{u}&=&0&\textrm{on }\partial D.
\end{array}
\right.
\end{equation}
Integrating the above expressions \eqref{eq:ran1}, \eqref{eq:ran2}, \eqref{eq:ran3} with respect to the random variable $\omega\in\Omega$ and setting
$$u(x)=\int_\Omega\tilde{u}(x,\omega)\,dP(\omega),\qquad
u_n(x)=\int_\Omega\tilde{u}_n(x,\omega)\,dP(\omega),$$
and $\displaystyle f(x)=\int_\Omega\tilde{f}(x,\omega)\,dP(\omega)$, one has
$$\left\{
\begin{array}{rcll}
-\dive\big(A_n(x)\nabla u_n(x)\big)&=&f(x)&\textrm{in }D\\
u_n&=&0&\textrm{on }\partial D.
\end{array}
\right.$$
and
$$\left\{
\begin{array}{l}
u_n\rightharpoonup u\quad\textrm{in }H^1_0(D),\\
A_n\nabla u_n\rightharpoonup A_\ast\nabla u\quad\textrm{in }L^2(D)^d,
\end{array}
\right.$$
where $u$ is the solution of the homogenized equation
$$\left\{
\begin{array}{rcll}
-\dive\big(A_\ast(x)\nabla u(x)\big)&=&f(x)&\textrm{in }D\\
u&=&0&\textrm{on }\partial D.
\end{array}
\right.$$
>From the generality of $\tilde{f}$ we can deduce that $A_\ast\in G_\theta$. \end{proof}

\subsection{Relaxation}

We now consider the classical optimal design problem

$$(O_c)\qquad\min I(\chi)=\int_\Omega\Big[\int_D
j(x,\omega,u_\chi(x,\omega))\,dx\Big]\,dP(\omega)$$
subject to
$$\begin{array}{c}
\chi\in L^\infty(\Omega;\{0,1\}),\textrm{ with }A=\alpha I_d\chi+\beta I_d(1-\chi),\\
-\dive\big(A(x)\nabla u(x,\omega)\big)=f(x,\omega)\quad\textrm{in }D,\\
u=0\quad\textrm{on }\partial D,
\end{array}$$
$P$-a.e. $\omega\in\Omega$, and to the volume constraint
$$\int_D\chi(x)\,dx\le L,$$
with $L\in(0,|D|)$.

The lack of optimal solutions for problems of the type ($O_c$) is well known even in the deterministic case (see \cite{Murat}).

The basic idea for the relaxation process consists in considering a larger class of admissible designs, in order a new (relaxed) problem on this larger class admits optimal solutions. Having in mind the above Theorem \ref{teo:carac_Gteta_autovalores}, we consider the space of generalized designs
$$\mathcal{GD}=\Big\{(\theta,A_\ast)\in L^\infty(D,[0,1])\times M^{d\times d}
\ :\ A_\ast\in G_{\theta(x)}\textrm{ a.e. }x\in D\Big\}.$$
Therefore we define the relaxed version ($O_r$) of the above optimal design problem as
$$(O_r)\qquad\min I(\theta,A_\ast)=\int_\Omega\Big[\int_D
j(x,\omega,u(x,\omega))\,dx\Big]\,dP(\omega)$$
subject to
\begin{equation}
\label{eq:estado}
\begin{array}{c}
\theta\in L^\infty(D;[0,1]),\textrm{ with }A_\ast\in G_\theta,\\
-\dive\big(A_\ast(x)\nabla u(x,\omega)\big)=f(x,\omega)\quad\textrm{in }D,\\
u=0\quad\textrm{on }\partial D,
\end{array}
\end{equation}
$P$-a.e. $\omega\in\Omega$, and the volume constraint
$$\int_D\theta(x)\,dx\le L.$$

\begin{teo}
($O_r$) is a relaxation of ($O_c$) in the sense that
\begin{enumerate}
\item the infima of both problems coincide
\item there are optimal solutions for the relaxed problem ($O_r$).
\end{enumerate}
\end{teo}

\begin{proof}
See for instance \cite{Allaire} Section 3.2, Theorem 3.2.1.
\end{proof}

\section{Optimal solutions}

In this section we consider the above problems ($O_c$) and ($O_r$) in the special case when the cost functionals are either the {\it compliance}
$$j(x,\omega,u)=f(x,\omega)u$$
or the {\it energy}
$$j(x,\omega,u)=-f(x,\omega)u$$
and we will prove that in these situations our original design problem
$$(O)\qquad\min I(a)=\int_\Omega\int_Dj(x,\omega,u_a(x,\omega))dx\,dP$$
subject to,
$$\begin{array}{c}
a\in L^\infty(D),\textrm{ with }\alpha\le a\le\beta,\\
-\dive\big(a(x)\nabla u(x)\big)=f\quad\textrm{in }D,\\
u=0\quad\textrm{on }\partial D,\\
\int_D a(x)\,dx=m
\end{array}$$
admits optimal solutions.

\begin{teo}\label{exist}
In the cases either of the compliance or of the energy the optimization problem \eqref{optpb} admits a solution.
\end{teo}

We know that the $(O_r)$ problem admits optimal solution (since is the relaxed problem of $(O_c)$), we check that these optimal solutions are solutions of our problem $(O)$ from which we deduce the well-posed character of our problem.

We analyze the optimality condition for $(O_r)$ for the matrix $A_\ast$. We denote by $p$ the adjoint state, which is the unique solution in $H^1_0(D)$ $P$-a.e. $\omega\in\Omega$ of the adjoint state equation
\begin{equation}\label{eq:coestado_relax_matrix}
\left\{\begin{array}{l}
\displaystyle-\dive\big(A_\ast(x)\nabla p(x,\omega)\big)=\frac{\partial j(x,\omega,u)}{\partial u}\quad\textrm{in }D,\\
p=0\quad\textrm{on }\partial D,
\end{array}\right.
\end{equation}
We remark that from \eqref{eq:coestado_relax_matrix} it follows that in the compliance case we have $p=u$ and in the energy case we have $p=-u$.

We fix the density $\theta\in L^\infty(D;[0,1])$ and we introduce the Lagrangian
\begin{eqnarray}\nonumber
\LL(M,\phi,\psi)=
\int_\Omega\int_Dj(x,\omega,\phi(x,\omega))\,dx\,dP(\omega)+\\
\int_\Omega\int_D\Big[\dive\big(M(x)\nabla\phi(x,\omega)\big)+
f(x,\omega)\Big]\psi\,dx\,dP(\omega)
\nonumber\end{eqnarray}
for any $M\in L^\infty(D;M^{d\times d})$ and $\phi,\psi\in H^1_0(D)$ $P$-a.e. $\omega\in\Omega$. We compute the partial derivative for any variable.

It is clear that $\big\langle\frac{\partial}{\partial\psi}\LL(M,\phi,\psi);\psi_1\big\rangle=0$ taking $\phi=u$ solution of \eqref{eq:estado}. We then
determine the solution $p$ so that, for all $\phi_1\in H^1(D)$ $P$-a.e. $\omega\in\Omega$, we have
$$\Big\langle\frac{\partial }{\partial\phi}\LL( M,\phi,p);
\phi_1\Big\rangle=0,$$
which leads to the formulation of the adjoint problem \eqref{eq:coestado_relax_matrix}.

Finally, from $I(\cdot,M)=\LL(M,u,p)$ it is easy to compute
$$\Big\langle\frac{\partial}{\partial M}\LL(M,u,p);M_1\Big\rangle
=-\int_\Omega\int_D M_1\nabla u\cdot\nabla p\,dx\,dP(\omega).$$
Therefore it is easy to deduce that our cost functional is G\^{a}teaux differentiable with respect to the matrix variable and its derivative in the direction $M_1$ is given by the above formula with $u$ and $p$ solutions of the state and adjoint equation respectively. Hence if $M_\ast$ if optimal, the optimality condition
$$\Big\langle\frac{\partial}{\partial M}
\LL(M_\ast,u,p);M-M_\ast\Big\rangle\ge 0\qquad\forall M\in G_\theta$$
becomes
\begin{equation}\label{def:opt_condi_punctual}
-\int_D\int_\Omega(M-M_\ast)\nabla u\cdot\nabla p\,dP(\omega)\,dx\ge0
\qquad\forall M\in G_\theta.
\end{equation}

>From the optimality condition \eqref{def:opt_condi_punctual} we obtain
\begin{equation}\label{def:opt_condi_max}
\int_D\int_\Omega M_\ast\nabla u\cdot\nabla
p\,dP(\omega)\,dx=\max_{M\in G_\theta}\int_D\int_\Omega M\nabla
u\cdot\nabla p\,dP(\omega)\,dx.
\end{equation}

Using the algebraic expression
$$4 M y\cdot z= M(z+y)\cdot(z+y)-M(z-y)\cdot(z-y),$$
we have that for any $M\in G_\theta$
$$4M\nabla u\cdot\nabla p\le
\max_{A\in G_\theta}A(\nabla u+\nabla p)\cdot(\nabla u+\nabla p)
-\min_{A\in G_\theta}A(\nabla u-\nabla p)\cdot(\nabla u-\nabla p)$$
and using the identification of $G_\theta$ we obtain
$$4M\nabla u\cdot\nabla p\le
\lambda^+_\theta|\nabla u+\nabla p|^2-\lambda^-_\theta|\nabla u-\nabla p|^2.$$
Having in mind that in our problem $u=p$ for the compliance and $u=-p$ for the energy, the necessary condition above reads
\begin{equation}\label{def:opt_condi_lambda}
\int_D\int_\Omega M_\ast\nabla u\cdot \nabla u\,dP(\omega)\,dx=
\int_D\int_\Omega\lambda_\theta|\nabla u|^2\,dP(\omega)\,dx
\end{equation}
with $\lambda_\theta=\lambda^+_\theta$ for the compliance and
$\lambda_\theta=\lambda^-_\theta$ for the energy.

If we analyze the optimality condition \eqref{def:opt_condi_lambda} we obtain that $\lambda_\theta$ is an eigenvalue of $M_\ast$ and $\nabla u(x,\omega)$ is an eigenvector $P$-a.e $\omega\in\Omega$, i.e.,
\begin{equation}\label{pro:optimo}
M_\ast\nabla u=\lambda_\theta\nabla u\qquad\textrm{$P$-a.e. }\omega\in\Omega
\end{equation}
where $u$ is the solution of the state equation.

%SOME MODIFICATIONS HERE.........

For the compliance case $u=p$ and
$$
M_\ast\nabla u\cdot\nabla u=
\lambda^+_\theta|\nabla u|^2,
$$
there exist several matrices in
$G_\theta$ with this property, and it is enough to take a rank one
laminate with normal direction of lamination $\vec{n}$ orthogonal
to $\nabla u$ and the optimal volume fraction $\lambda^+_\theta$ .

For the energy case $u=-p$ and
$$
M_\ast\nabla u\cdot\nabla u
=\lambda^-_\theta|\nabla u|^2,
$$
there exist an unique matrix in $G_\theta$ with this property
which corresponds to a rank one laminate with normal direction of
lamination $\vec{n}$ parallel to $\nabla u$ and the optimal volume
fraction $\lambda^-_\theta$.

We would like to stress that for any case the optimal matrix
$M_\ast$ is a first order laminate with deterministic optimal
volume fraction $\lambda_\theta$ and random direction of
lamination according with the random value of $\nabla u$.

Therefore, from the analysis of the optimality condition we concluded that the optimal matrix $M_\ast\in G_\theta$ verifies the condition \eqref{pro:optimo}. We remark that this condition does not imply that the optimal matrix $M_\ast$ is the isotropic matrix $\lambda_\theta I_d$; the important conclusion of \eqref{pro:optimo} is that $\lambda_\theta$ is an eigenvalue of $M_\ast$ and $\nabla u$ is an associated eigenvector, where $u$ is the solution of the state equation. In particular, this implies that the optimal value from the $(O_r)$ is attained on the simpler problem
$$(O)\qquad\min I(\theta)=\int_\Omega\int_Dj(x,\omega,u(x,\omega))dx\,dP(\omega)$$
subject to,
$$\begin{array}{c}
\theta\in L^\infty(D;[0,1]),\\
-\dive\big(\lambda_\theta(x)\nabla u(x)\big)=f\quad\textrm{ in }D,\\
u=u_0\quad\textrm{on }\partial D,\\
\int_\Omega\theta(x)\,dx=L
\end{array}$$
Finally, the proof of Theorem \ref{exist} reduces to the fact that for $\theta\in L^\infty(D;[0,1])$ one has that
$\lambda_\theta\in L^\infty(D;[\alpha,\beta])$, in particular it is enough to take $a=\lambda_\theta$ to deduce the existence of optimal solution of our original problem $(O)$.

\section{Numerical analysis of the optimal design problem}

We approach in this section the numerical resolution of the problem
$(O)$ for the compliance case for which $j(x,\omega,u)=u(x)f(x,\omega)$ with Dirichlet boundary condition, for which the cost can be written as
$$\int_\Omega\int_Da(x)|\nabla u(x,\omega)|^2\,dx\,dP(\omega).$$
We first describe an algorithm of minimization and then present some numerical
experiments. We treat the problem
$$(O)\qquad\min I(a)=\int_D\int_\Omega u(x)(f(x)+\xi(x,\omega))\,dP(\omega)\,dx$$
subject to,
\begin{equation}\label{eq:estado_relaxato}
\begin{array}{c}
a\in L^\infty(D),\textrm{ with }\alpha\le a\le\beta,\\
-\dive\big(a(x)\nabla u\big)=f+\xi\quad\textrm{in }D,\textrm{ $P$-a.e. }\omega\in\Omega\\
u=0\quad\textrm{on }\partial D,\\
\int_Da(x)\,dx=m
\end{array}
\end{equation}
where $\xi$ is a random variable $\xi=\xi(x,\omega)$. Similar computations are also made for the energy case $j(x,\omega,u)=-u(x)f(x,\omega)$.

\subsection{Algorithm of minimization}

We present the resolution of the optimal design problem $(O)$ using a gradient descent method. In this respect, we compute the first variation of the cost function with respect to $a$.

For any $\eta\in\R^+$, $\eta\ll 1$, and any $\tilde{a}\in L^\infty(D)$, we associate to the perturbation $a^{\eta}=a+\eta\tilde{a}$ of $a$ the derivative of $I$ with respect to $a$ in the direction $\tilde{a}$ as follows:
\begin{equation}
\frac{\partial I(a)}{\partial a}\cdot \tilde{a}
=\lim_{\eta\rightarrow 0} \frac{I(a+\eta \tilde{a})-I(a)}{\eta}.
\nonumber
\end{equation}

\begin{teo}\label{thderivative}
The first derivative of $I$ with respect to $a$ in any direction $\tilde{a}$ exists and takes the form
\begin{equation}\label{deriveE}
\frac{\partial I(a)}{\partial a}\cdot\tilde{a}=
-\int_D\tilde{a}\Big(\int_\Omega\nabla u\nabla p\,dP(\omega)\Big)\,dx,
\end{equation}
where $u$ is the solution of \eqref{eq:estado_relaxato} and $p$ is the solution in $H_0^1(D)$ of the adjoint problem
\begin{equation}\label{eq:estado_adjunto}
\left\{\begin{array}{ll}
-\dive\big(a(x)\nabla p\big)=f+\xi&\textrm{in $D$,\quad$P$-a.e. }\omega\in\Omega\\
\displaystyle p=0&\textrm{on }\partial D.
\end{array}\right.
\end{equation}
\end{teo}

\begin{proof}We introduce the Lagrangian
\begin{eqnarray}
\LL(a,\phi,\psi)=\int_D\int_\Omega\phi(x,\omega)(f(x)+\xi(x,\omega))\,dP(\omega)\,dx+\nonumber\\
\int_D\int_\Omega\Big[\dive\big(a(x)\nabla\phi(x,\omega)\big)+f(x)+\xi(x,\omega)\Big]\psi\,dx\,dP(\omega)
\nonumber
\end{eqnarray}
for any $a\in L^\infty(D;[\alpha,\beta])$ and $\phi,\psi\in H^1_0(D)$ $P$-a.e. $\omega\in\Omega$. We compute the partial derivative for any variable.

It is clear that $\big\langle\frac{\partial}{\partial\psi}\LL(a,\phi,\psi); \psi_1\big\rangle=0$ taking $\phi=u$ solution of \eqref{eq:estado_relaxato}. We then determine the solution $p$ so that, for all $\phi_1\in H^1(D)$, we have
$$\Big\langle\frac{\partial }{\partial\phi}\LL( a,\phi,p);
\phi_1\Big\rangle=0,$$
which leads to the formulation of the adjoint problem \eqref{eq:estado_adjunto}.

Finally, it is easy to compute
\begin{equation}\label{derivetotala}
\Big\langle\frac{\partial}{\partial a}\LL(a,u,p);\tilde{a}\Big\rangle=
-\int_D\tilde{a}\Big(\int_\Omega\nabla u \nabla p\,dP(\omega)\Big)\,dx.
\end{equation}
Next, writing that $I(a)=\LL(a,u,p)$, we obtain \eqref{deriveE} from \eqref{derivetotala}.\end{proof}

In order to take into account the volume constraint on $a$, we introduce the Lagrange multiplier $\gamma\in\R$ and the functional
$$I_\gamma(a)= I(a)+\gamma\int_Da(x)\,dx.$$
Using Theorem \ref{thderivative}, we then obtain easily that the first derivative of $I_\gamma$ is
$$\frac{\partial I_\gamma(a)}{\partial a}\cdot\tilde{a}
=-\int_D\tilde{a}\Big(\int_\Omega\nabla u \nabla p
\,dP(\omega)\Big)\,dx+\gamma\int_D\tilde{a}(x)\,dx,$$
which leads us to define the following descent direction:
\begin{equation}\label{descentdirection_a}
\tilde{a}(x)=\Big(\int_\Omega\nabla u \nabla p
\,dP(\omega)-\gamma\Big)\quad\forall x\in D.
\end{equation}
In this way, for any function $\eta\in L^\infty(D,\R^+)$ with $\|\eta\|_{L^\infty(D)}$ small enough, we have $I_\gamma(a+\eta\tilde{a})\le I_\gamma(a)$. The multiplier $\gamma$ is then determined so that, for any function
$\eta\in L^\infty(D,\R^+)$, $\|a+\eta\tilde{a}\|_{L^1(D)}=m$,
leading to
\begin{equation}\label{multiplier_a}
\gamma=\frac{\displaystyle\Big(\int_Da(x)dx-m\Big)+
\int_D\eta\int_\Omega\nabla u\nabla p\,dP(\omega)\,dx}
{\displaystyle\int_D\eta(x)\,dx}
\end{equation}
where the function $\eta$ is chosen such that $a+\eta\tilde{a}\in[\alpha,\beta]$ for all $x \in D$. A simple and efficient choice consists of taking $\eta(x)=\varepsilon (a(x)-\alpha)(\beta-a(x))$ for all $x\in D$ with $\varepsilon$ small and positive.

Finally we show the descent algorithm to solve numerically the optimization problem $(O)$.

We consider $0<\alpha\le\beta$, $m\in(\alpha|D|,\beta|D|)$ and $\varepsilon<1$, $\varepsilon_1\ll1$ data of the problem, the structure of the algorithm is as follows.
\begin{itemize}
\item Initialization of the density $a^0\in L^\infty(D;[\alpha,\beta])$;
\item for $k\ge0$, iteration until convergence (i.e., $|I_\gamma(a^{k+1})-I_\gamma(a^k)|\le\varepsilon_1|I_\gamma(a^0)|$) as follows:

\begin{itemize}
\item compute the solution $u_{a^k}$ of \eqref{eq:estado_relaxato} and then the solution $p_{a^k}$ of \eqref{eq:estado_adjunto}, both corresponding to $a=a^k$;
\item compute the descent direction $\tilde{a}$ defined by \eqref{descentdirection_a}, where the multiplier $\gamma$ is defined by \eqref{multiplier_a};
\item update the density $a^k$ in $D$:
$$a^{k+1}= a^k+\varepsilon(a^k-\alpha)(\beta-a^k)\tilde{a}^k,$$
with $\varepsilon\in\R^+$ small enough to ensure the decrease of the cost function, $a^{k+1}\in L^\infty(D,[\alpha,\beta])$.
\end{itemize}
\end{itemize}

\subsection{Numerical experiments}

In this section we implement the gradient descent algorithm explained in the previous subsection. These sort of problems have been studied by other authors (\cite{Allaire}, \cite{sigmund}, \cite{Casado_numeric}, \cite{Pgal_first_article}), and we will show the numerical results according with previous numerical simulations of the previous authors.

We solve the problem $(O)$ on the square domain $D=(0,1)^2$ for two phases $\alpha=1$ and $\beta=2$, the determinist part of the right-hand side of the state equation $f\equiv 1$ and we consider the volume constraint $m=\frac{\alpha+\beta}{2}=1.5$, i.e. we can use the same amount of $\alpha$ or $\beta$ mass. In order to simplify the numerical computations we choose the random variable $\xi$ with a discrete distribution of probability. We consider two different cases for $\xi$:
\begin{itemize}
\item\textbf{Case 1:} $\xi(x)=\pm\chi_{D_0}$ where $D_0=[\frac{1}{4},\frac{3}{4}]^2\subset D$\\
\item \textbf{Case 2:} $\xi(x)=\pm\chi_{D_1}$ where $D_1=D\setminus D_0$
\end{itemize}
and in both cases $P(\{\xi=\chi\})=P(\{\xi=-\chi\})=\frac{1}{2}$. It could be possible that the algorithm fall down into local minima of $I$, for this reason, we consider a constant initialization $a_0=m$, without any privilege for the optimal localization.

In the figures below we represent the corresponding optimal mass distribution $a$. We show numerical results for full deterministic case and the two random cases described above, all for the compliance and energy minimization. The result are qualitatively different for any case.

\begin{figure}[h!]
\begin{minipage}[t]{6.1cm}
\centering
\includegraphics[scale=0.53]{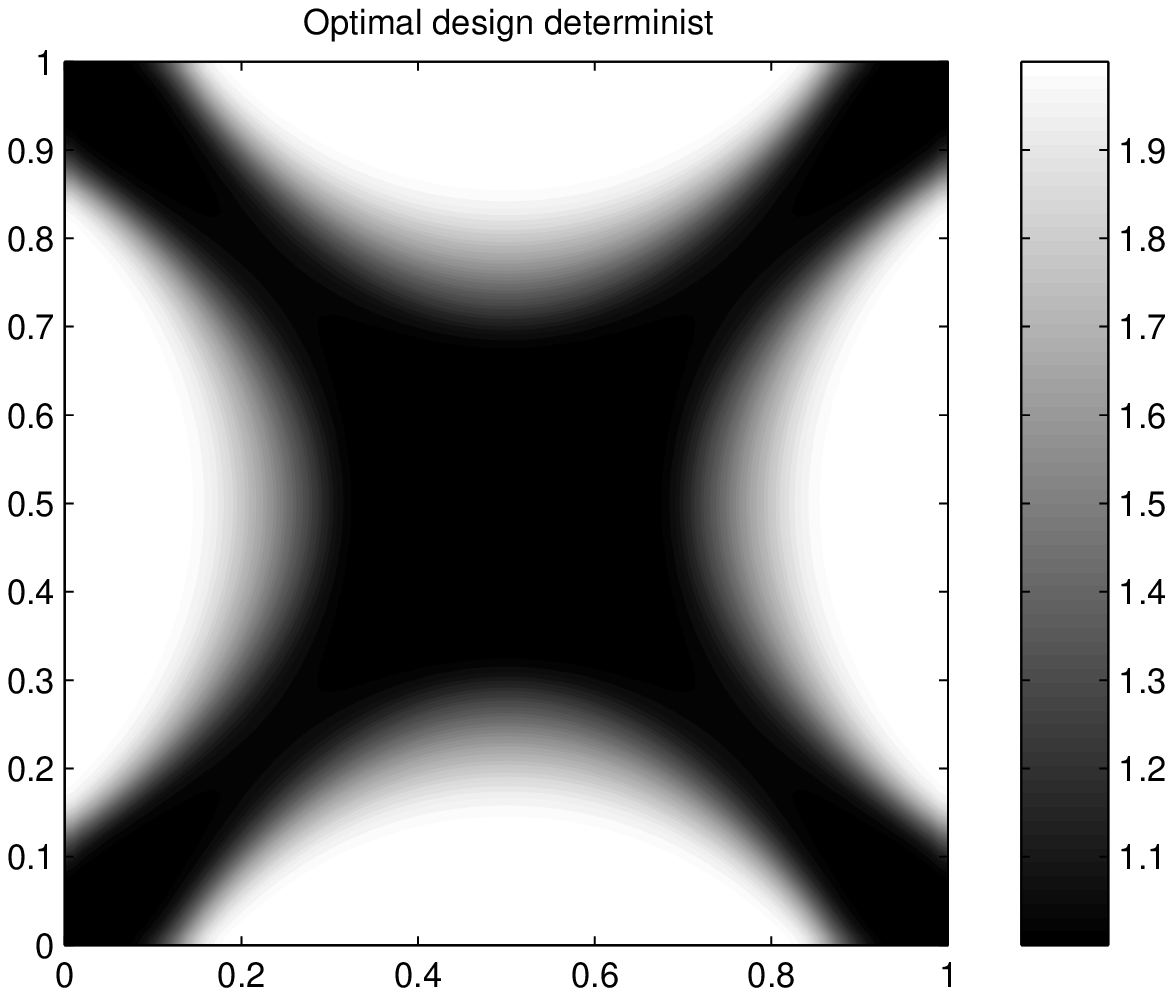}
\end{minipage}\hspace{0.1cm}
\begin{minipage}[t]{6.1cm}
\centering
\includegraphics[scale=0.53]{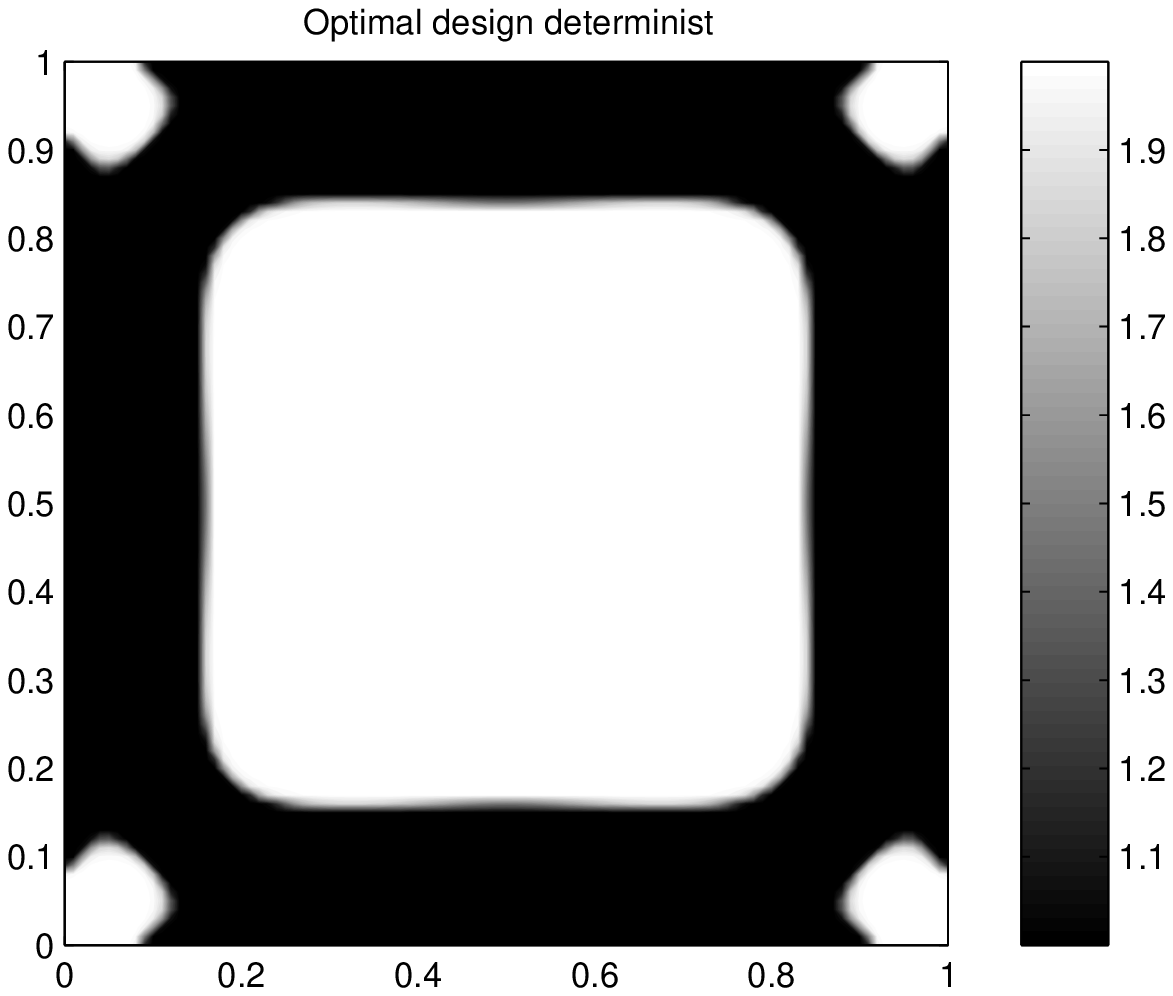}
\end{minipage}
\caption{Optimal distribution for the full deterministic case.
Left:  Compliance minimization. Right: Energy minimization.}
\label{Fig:deterministico}
\end{figure}

\begin{figure}[h!]
\begin{minipage}[t]{6.1cm}
\centering
\includegraphics[scale=0.53]{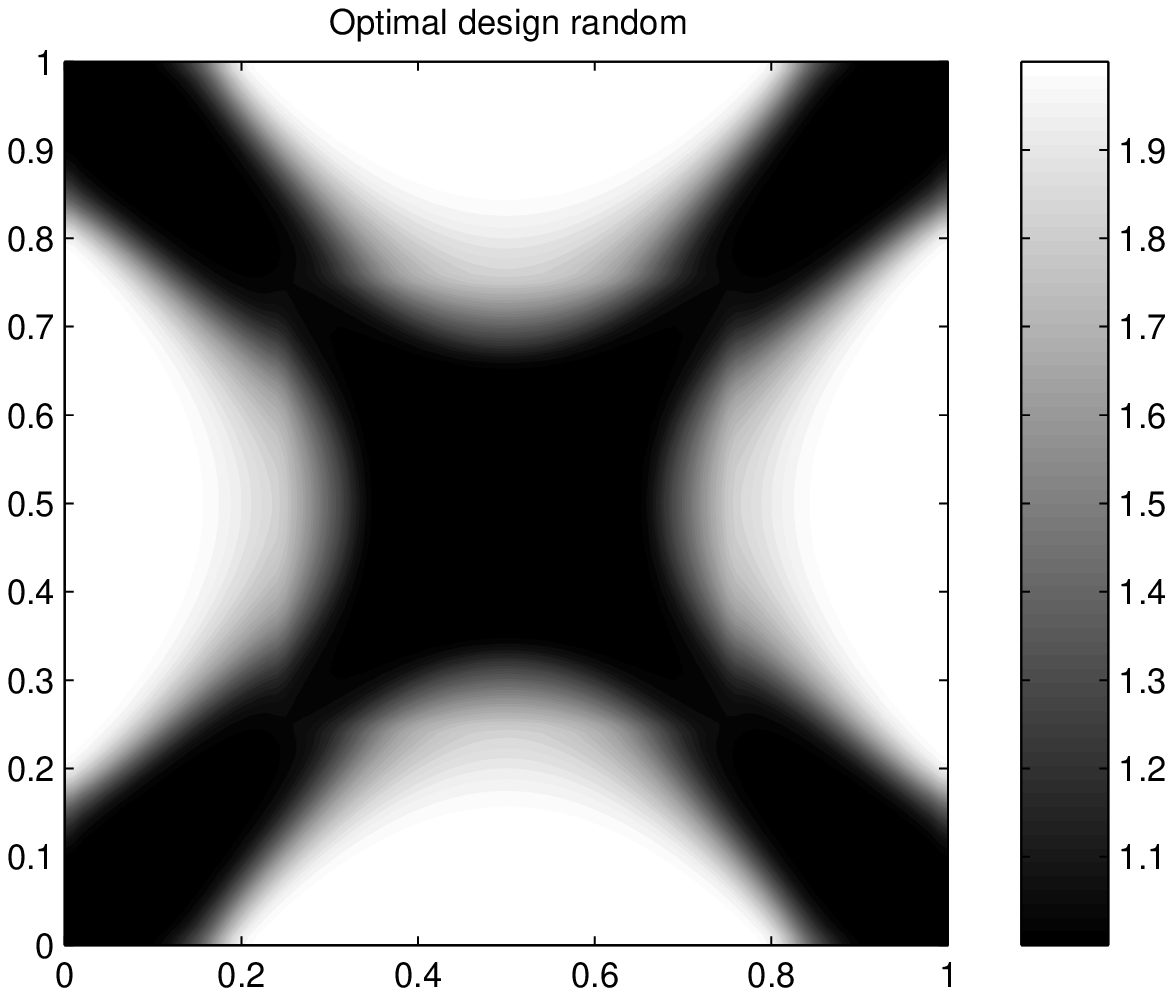}
\end{minipage}
\hspace{0.1cm}
\begin{minipage}[t]{6.1cm}
\centering
\includegraphics[scale=0.53]{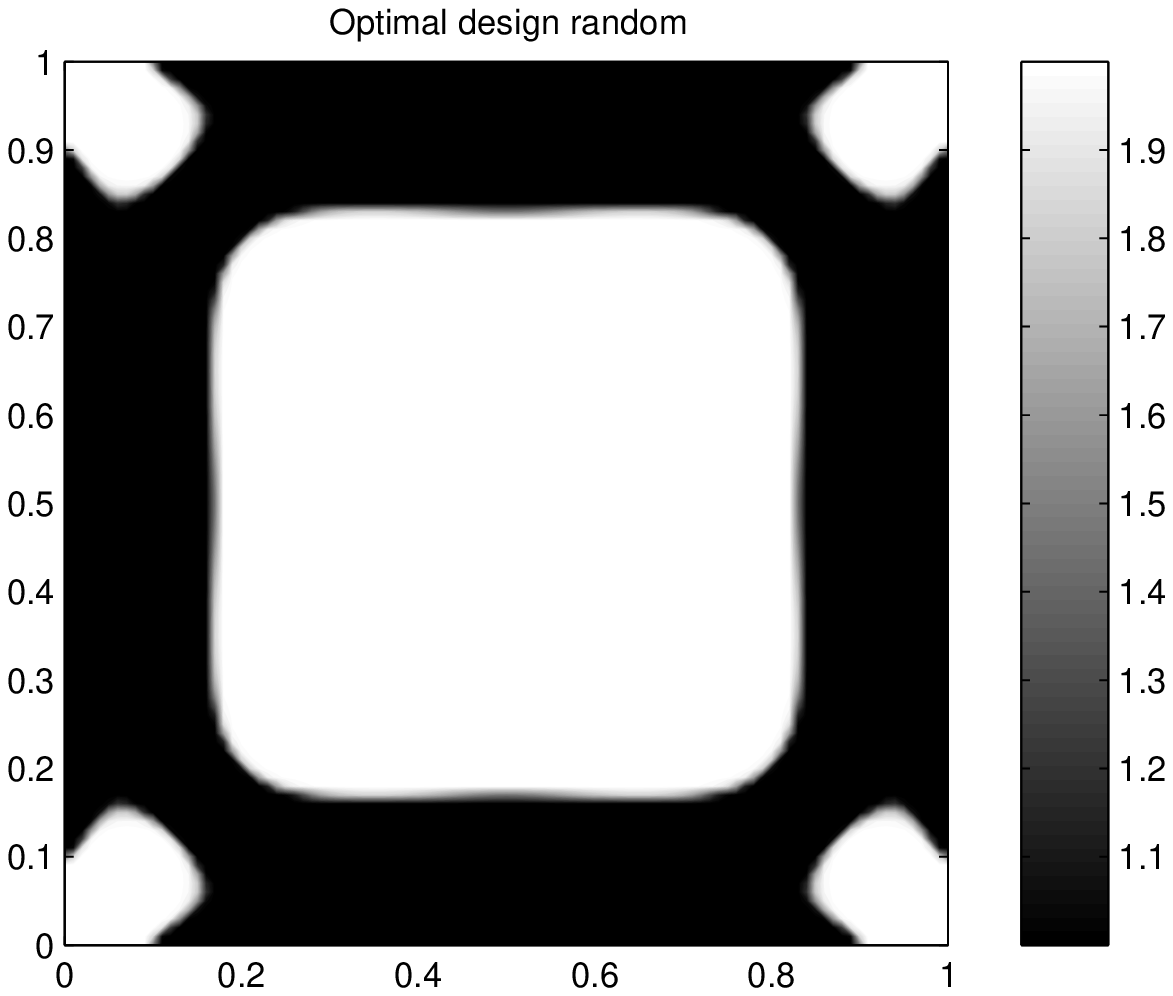}
\end{minipage}
\caption{Optimal distribution for the random Case 1. Left:
Compliance minimization. Right: Energy minimization}
\label{Fig:Estocas_case1}
\end{figure}

\begin{figure}[h!]
\begin{minipage}[t]{6.1cm}
\centering
\includegraphics[scale=0.53]{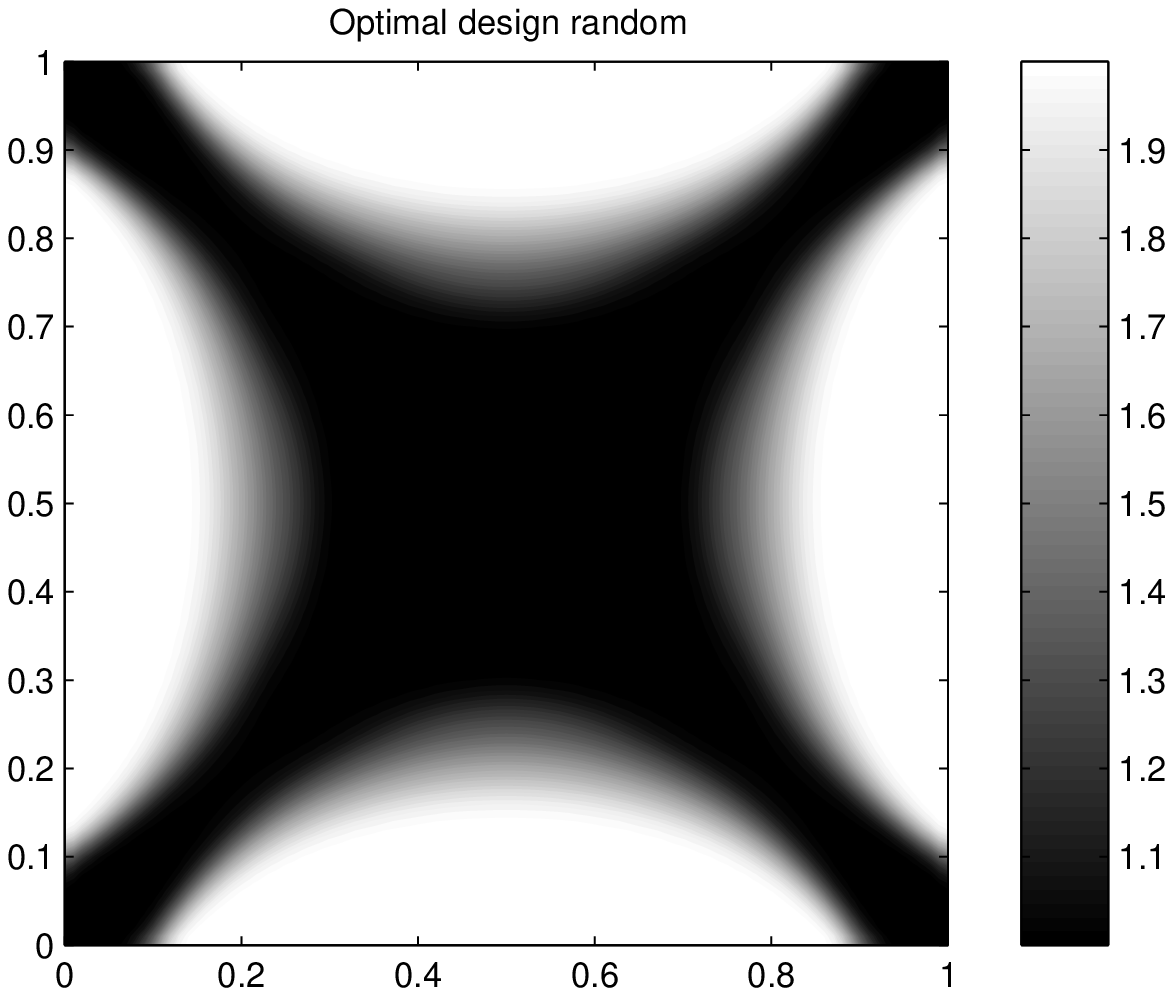}
\end{minipage}
\hspace{0.1cm}
\begin{minipage}[t]{6.1cm}
\centering
\includegraphics[scale=0.53]{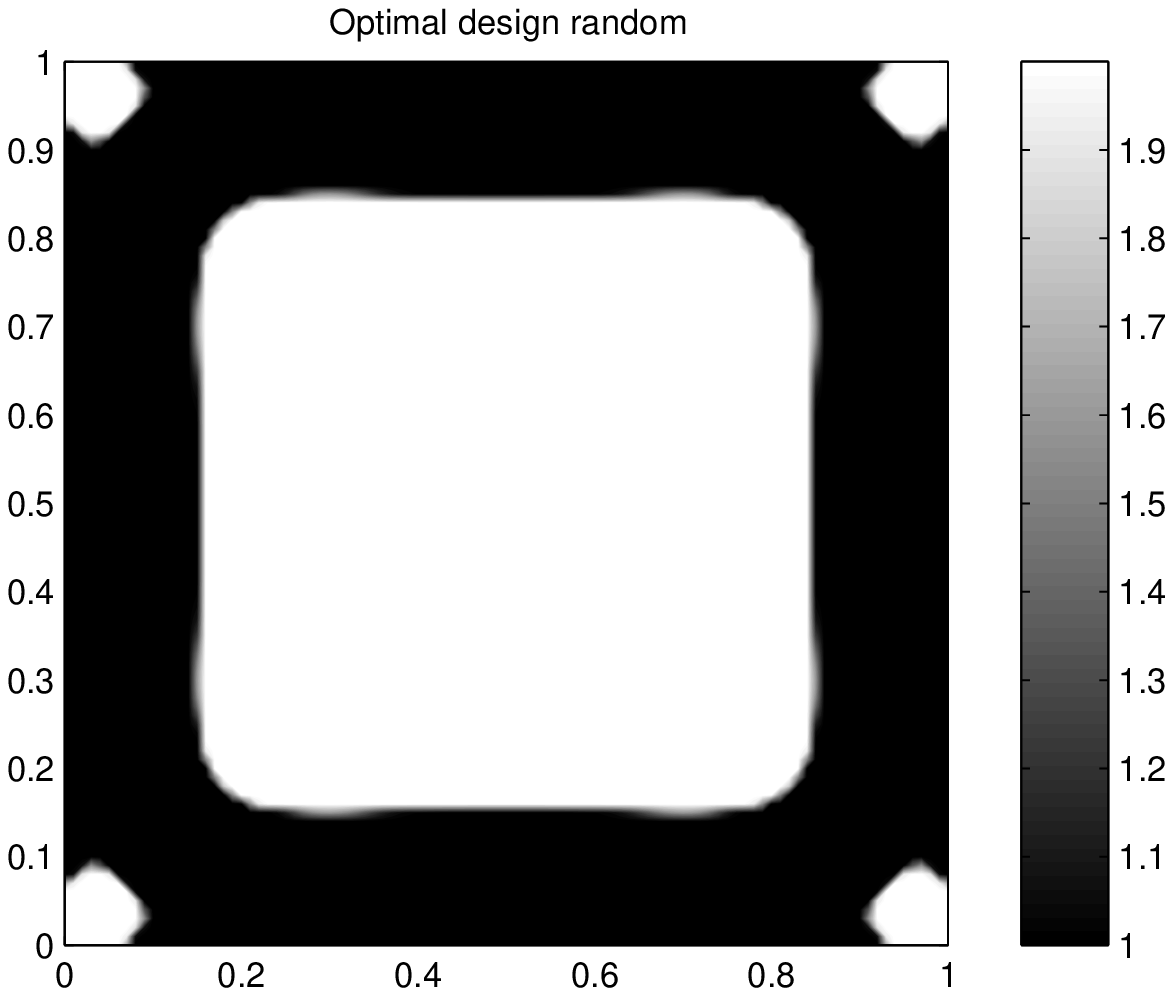}
\end{minipage}
\caption{Optimal distribution for the random Case 2. Left:
Compliance minimization. Right: Energy minimization}
\label{Fig:Estocas_case2}
\end{figure}

With respect to the compliance minimization, we observe that the limit densities follow a similar distribution,  where the smaller amount of mass is placed as a cross-shape. Taking as reference the picture of Figure \ref{Fig:deterministico}  (the deterministic case) we observe different densities for the two random cases. For the Case 1, where the random perturbation is at the middle of the square a bigger amount of mass (the white color at the pictures) is distributed for this place and the thickness of the cross is bigger at the corners (in order to save the volume constraint) (see Figure \ref{Fig:Estocas_case1}). For the Case 2 the effect is the inverse, in this case the random perturbation is near the boundary; the optimal distribution consists in a smaller amount of mass in the middle and a bigger concentration near to the boundary with a smaller thickness of the cross (see Figure \ref{Fig:Estocas_case2}).

For the energy minimization, the optimal distribution is fully different and it is in accord with previous analysis (see \cite{Allaire}, \cite{Donoso}). We take again the deterministic case as the reference design; for this case the simulations show that a bigger amount of mass is placed at the corners and at a
square in the middle of the domain of design. For the Case 1 the optimal distribution is very similar to the deterministic one, and the changes correspond with placing a bit more of mass at the corners of the domain. For the Case 2, the effect is the reverse, in this case there are less mass at the corners and almost all the mass is placed at a square in the middle of the
domain.

Finally, and in conclusion the numerical experiments indicate that under random forces the optimal distribution consists in placing a bigger amount of mass where this random force acts in, in order to take into account the possibility for the load to have a random variation.

\bigskip
{\bf Acknowledgements.} This paper was written during a visit of Faustino Maestre in Pisa with a fellowship of the Spanish government.

\end{document}